\documentclass[10pt]{article}

\usepackage{graphicx}
\usepackage{amsmath,amsfonts,amssymb,amsxtra,amsthm}
\usepackage{mathrsfs}
\usepackage{stmaryrd}
\usepackage{pinlabel}

\DeclareMathOperator{\Aut}{Aut} \DeclareMathOperator{\Out}{Out}
\DeclareMathOperator{\Hom}{Hom} 
 \DeclareMathOperator{\Comp}{C}

\DeclareMathOperator{\rk}{rank} \def\immerses{\looparrowright}
\def\onto{\twoheadrightarrow} \def\into{\hookrightarrow}
\def\fontop#1{\stackrel{#1}{\longrightarrow}}
\def\ncl#1{\mathord{\langle}\mskip -4mu plus 0mu minus 0mu
  \mathord{\langle}#1\mathord{\rangle}\mskip -4mu plus 0mu minus 0mu
  \mathord{\rangle}}

\def\zee{\mathbb{Z}}
\def\free{\ensuremath{\mathbb{F}}}

\def\define{\mathrel{\mathop:}=}

\newcommand{\mnote}[1]{}

\newtheorem{theorem}{Theorem}[section]

\newtheorem{lemma}[theorem]{Lemma}
\newtheorem{example}[theorem]{Example}

\theoremstyle{definition}
\newtheorem{definition}[theorem]{Definition}

\theoremstyle{remark}

\author{Larsen Louder} 

\title{Nielsen equivalence in closed surface groups}

\begin{document}
\maketitle

\begin{abstract}
  Any two generating systems of the fundamental group of a closed
  surface are Nielsen equivalent.
\end{abstract}

\section{Introduction}

The purpose of this paper is to show that there is essentially only
one way to generate the fundamental group of a closed surface. A
\emph{marking} of a finitely generated group $G$ is a surjective map
$f\colon\free_n\onto G$, where $\free_n$ is a free group of rank $n$
with a fixed basis $x_1,\dotsc,x_n$. Markings are tautologically in
one-to-one correspondence with tuples $(s_1,\dotsc,s_n)\in G^n$ such
that the elements $s_1,\dots,s_n$ generate $G$. Two markings are
\emph{Nielsen equivalent} if they differ by an automorphism of
$\free_n$, i.e., $f\sim g$ if there is some
$\varepsilon\in\Aut(\free_n)$ such that $f=g\circ\varepsilon$. By
Nielsen's theorem, $\Aut(\langle x_i\rangle)$ is generated by signed
permutations (maps of the form $x_i\mapsto x^{\pm 1}_{\sigma(i)}$,
where $\sigma$ is a permutation of $(1,\dotsc,n)$) and product
replacements (maps of the form $x_i\mapsto x_ix_j$, $x_k\mapsto x_k$
for $j,k\neq i$). Collectively these are known as Nielsen
transformations. Nielsen transformations descend to transformations on
tuples $(s_i)$ in the obvious way.

\begin{theorem}
  \label{maintheorem}
  Let $\Sigma$ be a closed surface with $\chi(\Sigma)\leq 0$. Then any
  two markings of $\pi_1(\Sigma)$ of the same length are Nielsen
  equivalent.
\end{theorem}

This is well known for the torus and Klein bottle, though our proof
also works for these two cases. Zieschang proved that all minimal rank
markings of closed orientable surface groups of genus $g\neq 3$ are
equivalent \cite{zieschang}.

A marking $f\colon\free_n\onto G$ is \emph{reducible} if $f(x_i)=1$
for some $i$, and \emph{weakly reducible} if $f$ factors as $g\circ
r$, where $r\colon\free_n\onto \free_{n-1}$ and
$g\colon\free_{n-1}\onto G$: there is a free basis $y_1,\dotsc,y_n$
of $\free_n$ such that $r(y_n)=1$ and $r(y_1),\dotsc,r(y_{n-1})$ is a
free basis of $\free_{n-1}$.  Likewise a tuple is reducible if $s_i=1$
for some $i$, and is weakly reducible if equivalent to a reducible
one.  Alternatively, say that $(s_1,\dotsc,s_n)$ is obtained from
$(s_1,\dotsc,\hat s_i,\dotsc,s_n)$ by \emph{stabilization} if
$s_i=1$. Likewise, if $s_{i_1},\dotsc,s_{i_k}$ are trivial, $i_p\neq
i_q$ for all $p$ and $q$, then $S$ is \emph{$k$ stabilized}.

It is often inconvenient to work directly with elements of a group. In
what follows we reinterpret Stallings/Bestvina, Feighn/Dunwoody
folding sequences for morphisms of graphs of groups in a category of
square complexes. In this setting folds of graphs of groups become
collapses/uncollapses of free faces and folds in the one-skeleton.

We finish this introduction with a brief overview of known results and
related topics. The oldest incarnation of Nielsen equivalence is the
Euclidean algorithm: if $p\geq q\geq 0$ generate $\zee$, then $p-q,q$
generate $\zee$ as well. If $p$ and $q$ generate $\zee$ then by
applying Nielsen moves according to the Euclidean algorithm we obtain
the generating tuple $(1,0)$. More generally, Gaussian elimination
amounts to the fact that any two markings of $\zee^n$ with the same
rank are equivalent. Throughout this paper we are implicitly using the
fact that any two length $m$ markings of $\free_n$ are equivalent
(Theorem~\ref{cor::nielsen}).

There is a related notion from $3$--manifold topology. Given two
Heegaard splittings of a closed orientable three-manifold, how many
handles must be added to each in order to make them equivalent? Given
a genus $g$ Heegaard splitting, the cores of the associated
handlebodies~\cite{hempel-3manifolds} are naturally generating systems
of rank $g$ of the fundamental group. Adding additional handles
corresponds precisely to stabilization of the natural generating
systems, and the number of stabilizations needed to obtain equivalent
generating systems is clearly a lower bound to the number needed to
obtain equivalent Heegaard splittings. Lustig and Moriah use this
principle to give examples of inequivalent Heegaard splittings of
certain Seifert fibered spaces, first by producing Nielsen
inequivalent minimal cardinality generating systems of their
associated Fuchsian groups~\cite{lustig-moriah}.

The kind of generating system associated to a Heegaard splitting was
used by Kapovich and Weidmann as a prototype for constructing
$n$--generator small-cancellation groups with generating systems that
are inequivalent after one
stabilization~\cite{kapweidsmallcancellation}.

The question of equivalence is also interesting for finite
groups. Dunwoody showed that if $G$ is a finite solvable group of rank
$n-1$ then any generating system with $n$ elements is
reducible~\cite{dunwoody::nielsen}. On the other end of the spectrum
we have the Wiegold conjecture, which says that any two markings of
cardinality $n\geq 3$ of a finite simple group should be Nielsen
equivalent. In particular any two generating systems with two elements
should be equivalent after one stabilization. See~\cite{lubotzky} for
a survey on $\Aut(\free_n)$ actions on $\Hom(\free_n,G)$ for more
general, but related, groups $G$.

\subsection*{Acknowledgments}

Thanks to Juan Souto for suggesting the problem. The author was
supported in part by EPSRC grant EP/D073626/2, NSF MSPRF DMS-0703658,
and NSF RTG DMS-0602191.

\section{Nielsen equivalence}

As noted above, it is inconvenient to work directly with elements of
a group. Let $G$ be a finitely generated group and let $Y$ be a
topological space with a fixed identification of $\pi_1(Y)$ with
$G$. Call a pair $(X,\varphi),$ where $X$ is a compact, aspherical,
two-dimensional CW--complex, $\pi_1(X)$ is free, and $\varphi\colon
X\to Y$ a continuous, $\pi_1$--onto map, a \emph{generating system}
for $\pi_1(Y)$. A morphism $\varepsilon\colon (X,\varphi)\to
(X',\varphi')$ is a continuous map $\varepsilon$ such that
$\varepsilon_*$ is surjective and $\varphi'\circ\varepsilon$ is
homotopic to $\varphi$. The \emph{rank} of $(X,\varphi)$ is the
minimal number of elements needed to generate $\pi_1(X)$.

A morphism is an \emph{equivalence} if it is a homotopy equivalence,
and is a \emph{reduction} if it is surjective but not injective on
fundamental groups. If $(X,\varphi)$ factors through a lower rank
$(X',\varphi')$ then $(X,\varphi)$ is \emph{reducible}: Let $\{x_i\}$
be a (any) basis for $\pi_1(X)$, and let $S=(\varphi_*(x_i))$. Then by
Nielsen's theorem $S$ is equivalent to some $S'$ containing the
identity element.

Let $X$ be a compact aspherical CW--complex with free fundamental
group, and choose for $X$ a fixed identification of $\pi_1(X)$ with
$\langle x_i\rangle$. Given a rank $n$ marking $f\colon\free_n\onto G$
there is always a map $\hat f\colon X\to Y$ such that $\hat
f_*=f$. Any other identification $\langle y_i\rangle\cong\pi_1(X)$
then gives an equivalent generating tuple.

It follows from basic algebraic topology that, given two equivalent
markings $f$, $g$, and spaces $X_f$ and $X_g$ as above, there is a
homotopy equivalence $\varepsilon\colon X_f\to X_g$ such that
$g\circ\varepsilon$ is homotopic to $f$ and $\hat g_*\circ
\varepsilon_*=\hat f_*$. Conversely, if there is such a homotopy
equivalence then all markings determined by $f$ and $g$ are
equivalent.

\begin{definition}
  Consider the usual presentations
  \[
  \langle x_1,y_1, \dotsc,x_g,y_g \mid
  \left[x_1,y_1\right]\dotsb\left[x_g,y_g\right]\rangle 
  \]
  and
  \[\langle x_1,\dotsc,x_n\mid x_1^2\dotsb x_n^2\rangle \]
  of orientable and nonorientable surface groups, respectively. The
  markings $(x_i,y_i)$ and $(x_i)$ are \emph{standard}.

  A \emph{standard generating system} of a closed connected surface
  $\Sigma$ is a pair $(X,i)$, where $X$ is any compact CW strong
  deformation retract of $\Sigma\setminus p$ with $p\in\Sigma$ and
  $i\colon X\into\Sigma$ the inclusion map.
\end{definition}

\begin{lemma}
  All standard generating systems are equivalent.
\end{lemma}

\begin{proof}
  Let $(X,i)$ and $(X',i')$ be standard generating systems, and let
  $r\colon \Sigma\setminus p\to X$ and $r'\colon\Sigma\setminus p'\to
  X'$ be retractions associated to $X$ and $X'$. Let
  $\sigma\colon\Sigma\times\left[0,1\right]\to\Sigma$ be an isotopy of
  $\Sigma$ carrying $p$ to $p'$. (The ``point-pushing'' map.) Set
  $\varepsilon=r'\circ\sigma\vert_{\Sigma\times\{1\}} \circ i\colon
  X\to X'$. Then $i'\circ \varepsilon \sim_h i$.
\end{proof}

Consider a presentation $2$--complex homeomorphic to a surface
$\Sigma$ associated to one of the above presentations. In each case
the one-skeleton is an embedded subgraph $X\subset\Sigma$ representing
the standard markings $(x_i,y_i)$ (or $(x_i)$), with the property that
if $p\in \Sigma\setminus X$ then $X$ is a strong deformation retract
of $\Sigma\setminus p$. Note that $(X,i)$, where $i$ is the inclusion
map, is standard. In topological terms, Theorem~\ref{maintheorem} is
equivalent to:

\begin{theorem}
  \label{maintheoremii}
  Every generating system for the fundamental group of a
  closed surface $\Sigma$ with $\chi(\Sigma)\leq 0$ is either
  reducible or equivalent to a standard generating system.
\end{theorem}

\section{Graphs and Stallings' folding}

There are two commonly used definitions of graphs, each of which will
be useful. First, one may take a graph to be a one-dimensional cell
complex. Alternatively, a graph is a pair of sets $V$ and $E$, which
are the vertices and oriented edges, respectively, with a pair of maps
$\tau\colon E\to V$ and $\iota\colon E\to V$, with a fixed point free
involution $\overline{\phantom{e}}\colon E\to E$ such that
$\tau\circ\overline{\phantom{e}}=\iota$ and
$\iota\circ\overline{\phantom{e}}=\tau$~\cite[\S2.1]{serre}. Let $E'$
be collection of orbit representatives for
$\overline{\phantom{e}}$. Then a graph in the first sense may be
recovered by introducing a vertex for each element of $V$ and for each
element $e$ of $E'$ an edge $e$ identified with
$\left[0,1\right]$. Then glue $1$ to $\tau(e)$ and $0$ to
$\iota(e)$. It is often useful to think of a graph in the second sense
as a category. We will move freely between the two notions.

A \emph{Stallings fold, or simply fold,} is a surjective morphism of
graphs $\phi\colon V\to V'$ such that $\phi$ only identifies two edges
which have a common endpoint. A morphism of graphs induces maps of
links of vertices, and a morphism inducing injections on all links of
vertices is an \emph{immersion}. Immersions of graphs are
$\pi_1$--injective. A graph is \emph{core} if it doesn't have any
valence one vertices.

\begin{theorem}[\cite{stallings}]
  \label{stallings}
  A morphism of finite graphs $\varphi\colon V\to V'$ factors as 
  \[V=V_0\fontop{\phi_0}V_1\fontop{\phi_1}V_2\to\dotsb\to V_k\immerses V'\]
  where $\phi_i$ is a fold and $V_k\immerses V'$ is an immersion.
\end{theorem}

This is the first step in a geometric proof of the essential fact that
any two markings of the same length of a free group are Nielsen
equivalent. One can also give a very transparent proof, for example,
that $\Out(\free_n)$ is generated by Whitehead moves. See~\cite{wade}.

\begin{theorem}[Nielsen's theorem]
  \label{cor::nielsen}
  Let $f\colon\free_m\onto\free_n$ be a marking of $\free_n$ and let
  $x_1,\dotsc,x_n$ be a free basis of $\free_n$. Then there is a free
  basis $y_1,\dotsc,y_m$ of $\free_m$ such that $f(y_i)=x_i$ for
  $i\leq n$ and $f(y_i)=1$ for $n+1\leq i\leq m$.
\end{theorem}

The basic principle is that folds which are homotopy equivalences
correspond to Whitehead moves (which are compositions of Nielsen
transformations), and folds which aren't homotopy equivalences clearly
kill basis elements. The proof of Theorem~\ref{maintheorem} follows
the same lines. A graph is a graph of groups with trivial edge groups,
and Stallings' folding is a special case of a more general theory of
foldings of graphs of groups decompositions. Surfaces admit graphs of
groups decompositions over cyclic edge groups and free vertex
groups. Roughly, we represent a marking
$f\colon\free_n\onto\pi_1(\Sigma)$ as a morphism of graphs of groups
and construct a folding sequence which eventually demonstrates that
$f$ is either reducible or equivalent to a standard set of
generators. See Section~\ref{sec::folding}.

We leave the proof of the Lemma~\ref{starslemma}, needed in
Lemma~\ref{lem::notinjective}, to the reader. See
Figure~\ref{fig::continuefolding}.

\begin{lemma}
  \label{starslemma}
  Let $\phi\colon V\to V'$ be a morphism of finite core graphs which
  is surjective on links of vertices and not injective on at least
  one link. Then $\phi$ is not $\pi_1$--injective.
\end{lemma}

\begin{figure}[ht]
  \centerline{
    \includegraphics[width=.6\textwidth]{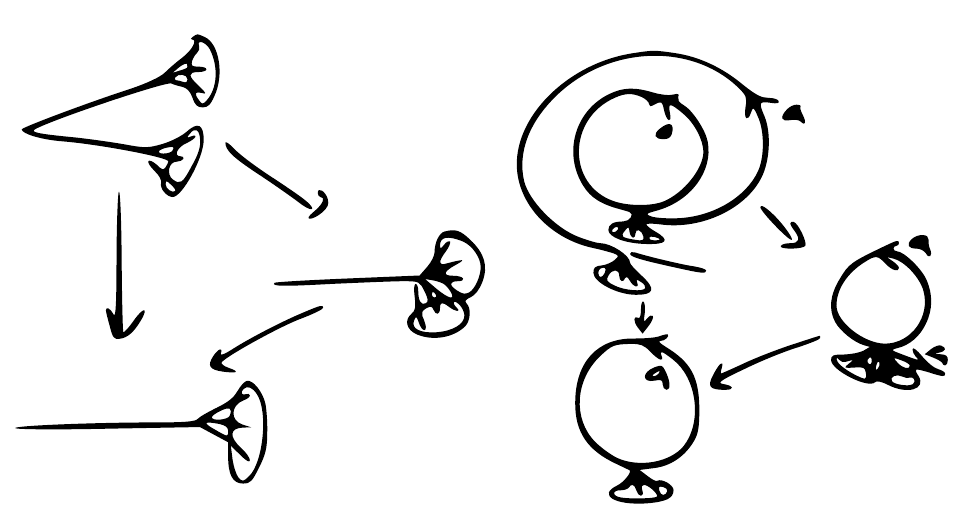}
  }
  \caption{A morphism of graphs which is surjective on all links of
    vertices can be folded again unless it identifies two edges with
    the same endpoints.}
\label{fig::continuefolding}
\end{figure}

\begin{lemma}
  \label{lem::notinjective}
  Let $A$, $B$, $C$, $V$ and $W$ be pointed core graphs with immersions of
  pointed graphs $\alpha\colon A\to C$, $\beta\colon B\to C$,
  $\gamma\colon C\to W$, $\epsilon\colon A\to V$, and $\nu\colon V\to
  W$, such that
  \begin{itemize}
    \item $A,B$, and $C$ are circles
    \item there is an edge $h$ of $V$ which is traversed exactly once
      by $\epsilon$
    \item $\nu$ is a finite sheeted cover and
    \item $\gamma\circ\alpha=\nu\circ\epsilon$.
  \end{itemize}
  Let $b$ and $v$ be the basepoints of $B$ and $V$. If
  $\deg(\beta)<\deg(\alpha)$ then the induced map
  \[
  V'=(V\setminus h^{\circ})\vee_{b=v}B \to W
  \]
  is not $\pi_1$ injective.
\end{lemma}

Intuitively, $A$ and $B$ are powers of $C$. Removing the interior of
$h$ eliminates $A$, but we add a lower power of $C$ by wedging
$V\setminus h^{\circ}$ and $B$ together. The map on $V'$ is not
injective on fundamental group. See Figure~\ref{fig::addroot}

\begin{proof}
  Let 
  \[
  V'\to D_1\to D_2\to\dotsb\to W
  \]
  be a folding sequence for the morphism $V'\to W$. For some $i$ the
  map $D_i\to W$ is surjective but not injective on all links of
  vertices. By Lemma~\ref{starslemma} $V'\to W$ is not injective on
  fundamental group.
\end{proof}

\begin{figure}[ht]
  \centerline{
    \includegraphics[width=\textwidth]{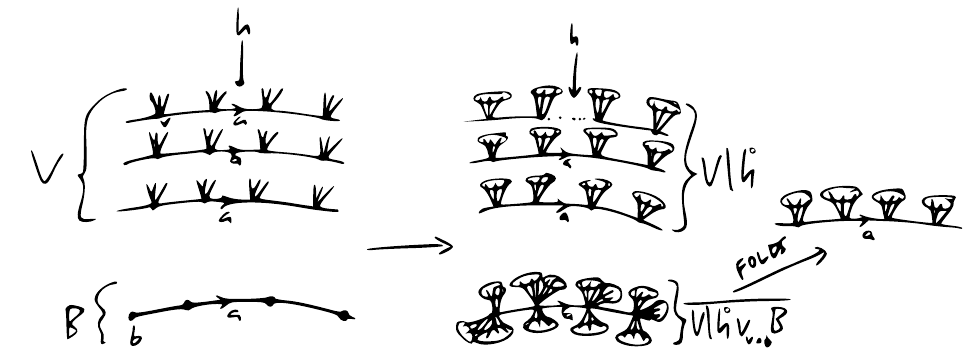}
  }
  \caption{$\overline{(V\setminus h^{\circ})\vee_{b=v}B}$ indicates a
    quotient of $V'$ where Lemma~\ref{starslemma} is applicable.}
\label{fig::addroot}
\end{figure}

\begin{definition}[Collapsible]
  \label{def::collapsible}
  Let $(V,\mathcal{E})$ be a pair consisting of a graph $V$ and a
  collection of immersed loops $\mathcal{E}=\{\tau_i\colon E_i\to
  V\}_{i\in I}$ with the property that every edge is traversed at most
  two times. Such a pair $(V,\mathcal{E})$ is \emph{collapsible} if
  and only if, for every nonempty $J\subset I$, there is an edge $e$
  in the subgraph $W=W_J=\cup_{j\in J}\tau_j(E_j)$ which is traversed
  exactly once by the collection $\{\tau_j\}_{j\in J}$.
\end{definition}

The following is an exercise.

\begin{lemma}
  \label{lem::foldcollapsible}
  Let $(V,\mathcal{E})$ be as above, let $\phi\colon V\to V'$ be a
  fold and let $\mathcal{E}'$ be the induced collection
  $\{\tau'_i=\phi\circ\tau_i\colon E_i\to V'\}$. Assume that each
  $\tau'_i$ is an immersion and that every edge of $V'$ is traversed
  at most twice. If $(V,\mathcal{E})$ is collapsible and $\phi$ is a
  homotopy equivalence then $(V',\mathcal{E}')$ is collapsible.
\end{lemma}

We briefly recall the construction of the pullback of a pair of
morphisms of graphs from~\cite{stallings}. Let $\alpha\colon A\to C$
and $\beta\colon B\to C$ be immersions of graphs. Then $\alpha$ and
$\beta$ represent conjugacy classes of subgroups of $\pi_1(C)$. To
compute the intersection of $A$ and $B$ form pullback of $\alpha$
and $\beta$:
\[A\times_{C}B=\{(x,y)\mid\alpha(x)=\beta(y)\}\subset A\times B\] Let
$\pi_A\colon A\times_CB\to A$ and $\pi_B\colon A\times_CB\to B$ be the
maps induced by the projections $A\times B\to A$ and $A\times B\to
B$. If $a$ and $b$ are basepoints in $A$ and $B$, respectively, then
the component of $A\times_CB$ containing $(a,b)$ represents the
intersection $\alpha_*(\pi_1(A,a))\cap\beta_*(\pi_1(B,b))$. Given a
pair of morphisms of graphs $\tau_A\colon D\to A$ and
$\varphi_B\colon D\to B$ such that
$\alpha\circ\tau_A=\beta\circ\varphi_B$ there is always a
\emph{unique} morphism of graphs $\epsilon\colon D\to A\times_{C}B$
such that $\pi_A\circ\epsilon=\tau_A$ and
$\pi_B\circ\epsilon=\varphi_B$.

\section{Folding graphs of groups}

\label{sec::folding}

A \emph{graph of groups} is a tuple $\Delta=(V,\{G_v\},\{G_e\})$,
where $V$ is a (finite) graph equipped with a \emph{vertex group}
$G_v$ for each vertex $v$, an \emph{edge group} $G_e$ for each edge
$e$, and for each endpoint $v$ of an edge $e$, an injective map
$G_e\into G_v$. A morphism of graphs of groups is a morphism of graphs
equipped with additional data corresponding to vertex and edge
groups. If $\Delta$ and $\Delta'$ are graphs of groups, then a
morphism $\varphi\colon\Delta\to\Delta'$ is given by a map
$\varphi\colon V\to V'$, and homomorphisms $\varphi_e\colon G_e\to
G'_{\varphi(e)}$, $\varphi_v\colon G_v\to G'_{\varphi(v)}$, such that
for each edge $e$ and vertex $v$ adjacent to $e$, the compositions
$G_e\into G_v\to G_{\varphi(v)}$ and $G_e\to G'_{\varphi(e)}\into
G'_{\varphi(v)}$ are equal up to conjugacy.

Every morphism of graphs of groups with finitely generated edge groups
factors through a sequence composed of three special kinds of
morphisms which are collectively called folds, each of which measures,
to some extent, a failure of Britton's lemma. \emph{Folding and
  pulling} are illustrated in Figures~\ref{fig::folding}
and~\ref{fig::pulling2}, which we have essentially borrowed
from~\cite{bf::folding}. A \emph{vertex morphism} is a morphism of
graphs of groups where the map is given by passing to a quotient of a
vertex group~\cite{dunwoody}.

Given an amalgamated product of groups (or HNN extension) $G=A*_CB$,
any $g\in G$ is either conjugate into one of $A$ or $B$, or has a
conjugate which can be written as a product $\Pi a_i b_i$ such that
$a_i, b_j\not\in C$. Likewise, if $G$ is the fundamental group of a
graph of groups $\Delta$ then any element has an essentially unique
normal form with respect to $\Delta$. If $G\to H$ is induced by a
morphism of graphs of groups $\Delta\to\Delta'$ and $g$ is written as
a reduced product of elements of vertex groups and stable letters,
then the image of $g$ in $H$ is given an induced description as a
product of elements of vertex groups and stable letters. If the
induced product is not reduced, then $\Delta\to \Delta'$ factors
through an elementary fold.

If $\phi\colon\Delta_G\to\Delta_H$ is a morphism of graphs of groups
then there is a $G$--equivariant map $\widetilde{\phi}\colon T_G\to
T_H$, where $T_G$ and $T_H$ are the trees associated to $\Delta_G$ and
$\Delta_H$. If $\phi$ is either folding or pulling then it corresponds
to a $G$--equivariant composition of folds which, in the former
case, identify edges in different $G$ orbits, and in the latter case,
identify edges in the same $G$--orbit. Vertex morphisms correspond to
dividing vertex stabilizers in $T_G$ by subgroups of the kernels of
their actions on $T_H$. Nevertheless, we call all three types of moves
folds. The map $\phi$ is simply a more compact description of
$\widetilde{\phi}$.

\begin{figure}[ht]
  \centerline{\includegraphics[width=.8\textwidth]{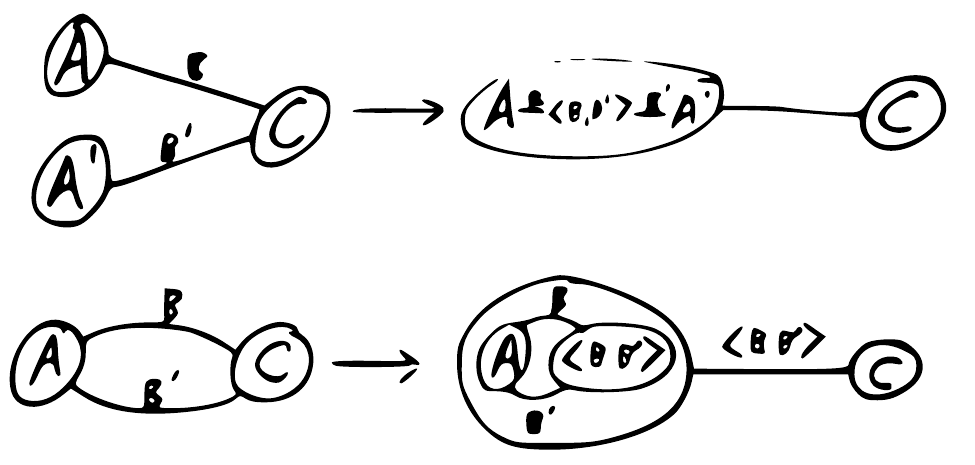}}
  \caption{Folding distinct edges together. We omit illustrating the
    case of a single vertex. In the first row the edges labeled $B$
    and $B'$ are folded together: The vertices labeled $A$ and $A'$
    are replaced by a single vertex labeled $A*_B\langle B,B'\rangle
    *_{B'}A'$ and an edge labeled $\langle B,B'\rangle$ is
    introduced.}
  \label{fig::folding}
\end{figure}

\begin{figure}[ht]
  \centerline{\includegraphics[width=.8\textwidth]{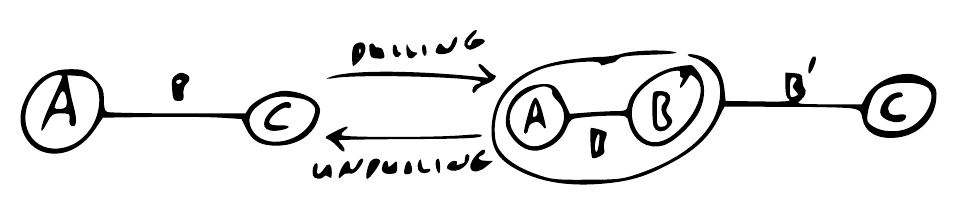}}
  \caption{Enlarging an edge group by pulling. In this case $B<B'<C$
    and we replace $A$ by $A*_BB'$ and $B$ by $B'$. The reverse is
    \emph{unpulling}}
  \label{fig::pulling2}
\end{figure}

Given a surface $\Sigma$ and a surjective map $f\colon \free_n\onto
G=\pi_1(\Sigma)$ we would like to find a folding sequence of graph of
groups decompositions
\[
\Delta_0\to \Delta_1\to\dotsb\to\Delta_m\to\Delta_{\Sigma}
\]
such that 
\begin{itemize}
\item $\Delta_0$ is a graph of groups decomposition of $\free_n$,
  $\Delta_{\Sigma}$ is a graph of groups decomposition of
  $\pi_1(\Sigma)$,
\item $\Delta_i\to\Delta_{i+1}$ is surjective on fundamental group,
  and
\item the fundamental group of $\Delta_i$ is free,
\item the terminal map $\Delta_m\to\Delta_{\Sigma}$ represents a
  standard generating system.
\end{itemize}

Existence of such a sequence would imply
Theorem~\ref{maintheorem}. Unfortunately, a naive folding sequence is
unlikely to satisfy the last two bullets, as the following example,
illustrated in in Figure~\ref{fig::naive}, shows.

\begin{example}
\label{ex::naive}
Let $G=\langle x_1,y_1,x_2,y_2\mid
\left[x_1,y_1\right]=\left[x_2,y_2\right]\rangle$ be the fundamental
group of the genus two surface, and let $A_1$ and $B_1$ be index three
and five subgroups of $\langle x_1,y_1\rangle$ and $\langle
x_2,y_2\rangle$, respectively, with one boundary subgroup,
corresponding to three and five-fold covers of the torus with
boundary. Then $A_1*B_1$ corresponds to a generating system of $G$
with ten generators. We define a folding sequence carrying $A_1*B_1$
to $G$ as follows:

We first pull the group
$\langle\left[x_2,y_2\right]^5\rangle=\langle\left[x_1,y_1\right]^5\rangle$
across the trivial edge in the free product $A_1*B_1$. Define
$A_2=A_1*\langle \left[x_1,y_1\right]^5\rangle$. Then $A_1*B_1\onto
A_2*_{\langle\left[x_1,y_1\right]^5\rangle}B_2$, with
$B_2=B_1$. Follow this by a vertex morphism $A_2\onto A_3$. Since
$A_2\to \langle x_1,y_1\rangle$ is surjective we have $A_3=\langle
x_1,y_1\rangle$. There is a natural map
$A_2*_{\langle\left[x_1,y_1\right]^5\rangle}B_2\onto
A_3*_{\langle\left[x_1,y_1\right]^5\rangle} B_3$, with $B_3=B_2$. Now
pull the subgroup $\langle\left[x_1,y_1\right]\rangle<A_3$ across the edge
to obtain an intermediate group $\langle
x_1,y_1\rangle*_{\langle\left[x_1,y_1\right]\rangle}B_4$, where
$B_4=(B_1*\langle\left[x_1,y_1\right]\rangle)/\ncl{\left[x_1,y_1\right]^5=\left[x_2,y_2\right]^5}$. Follow
this by a vertex morphism $B_4\onto B_5$. Since $B_4\to\langle
x_2,y_2\rangle$ is surjective, we have $B_5=\langle x_2,y_2\rangle$,
and the folding sequence is complete.
\end{example}

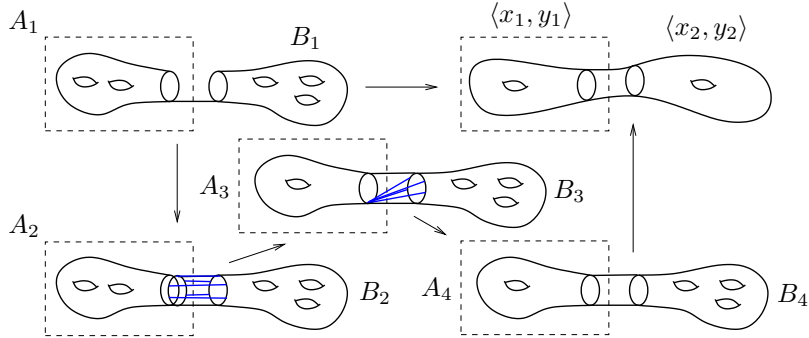
\begin{figure}
\labellist
\pinlabel $A_1$ [br] at 1 205
\pinlabel $B_1$ [b] at  179 192
\pinlabel $\langle x_1,y_1\rangle$ [b] at 333 206
\pinlabel $\langle x_2,y_2\rangle$ [b] at 453 197
\pinlabel $A_2$ [br] at  1 65
\pinlabel $B_2$ [l] at 210 30
\pinlabel $A_3$ [r] at 133 103
\pinlabel $B_3$ [l] at 343 100
\pinlabel $A_4$ [r] at 285 33
\pinlabel $B_4$ [l] at 496 29
\endlabellist
\centerline{
  \includegraphics[width=.8\textwidth]{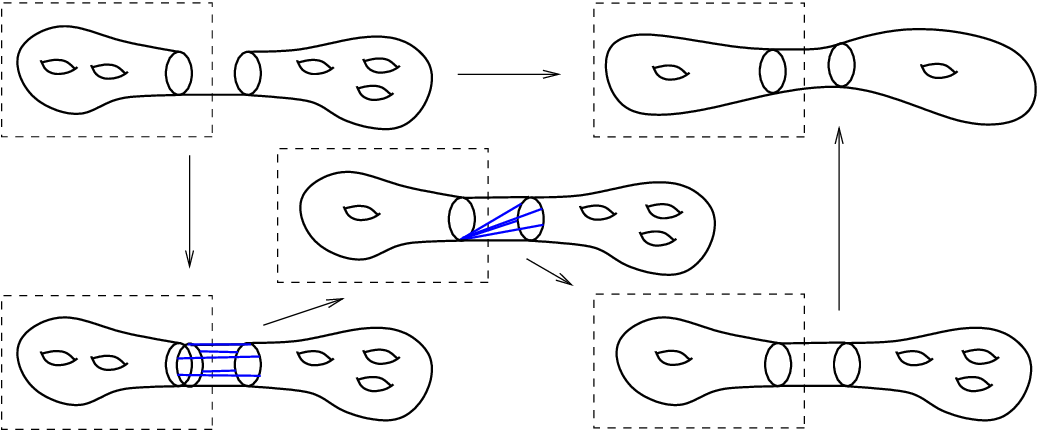}
}
\caption{The folding sequence for Example~\ref{ex::naive}}
\label{fig::naive}
\end{figure}

Note now that any sequence of this type must at some point give an
intermediate group that is not free. For instance, in the above
example, the map $A_2\onto A_3$ is a noninjective surjection of free
groups, and a primitive element $g$ in $A_2$ dies under the
map. However, it isn't clear from this sequence which, if any,
primitive elements in $A_1*B_1$ die under the map $A_1*B_1\onto G$:
the map $A_1*B_1\onto G_3$ doesn't visibly kill anything primitive
\emph{in the whole group}. This kind of folding sequence is
essentially the kind that Zieschang uses to prove
Theorem~\ref{maintheorem} for minimal generating systems: once a
folding sequence arrives at a free product of finite index subgroups
of vertex groups the minimality hypothesis guarantees that the
generating system is of the correct type. As we will see the way
around this is to allow unfoldings of graphs of groups, and to use a
target graph of groups decomposition which induces only HNN extensions
and not amalgamated products.

Rather than work directly with graphs of groups decompositions we
introduce a class of square complexes which provide geometric
representations of graphs of groups decompositions of generating
systems of closed surfaces. On this class we define a collection of
moves which correspond to folding, pulling, and vertex morphisms (and
some moves which do nothing to the associated graph of groups
decomposition). By applying the moves carefully we will build a
sequence of spaces which more and more closely approximate standard
generating systems. After each move we obtain a generating system
which is either a reduction of or equivalent to the system we started
with. Unlike in the Stallings/Bestvina, Feighn/Dunwoody scheme for
folding graphs of groups we must at some points perform moves which
correspond to unpulling edge groups. With respect to the correct
complexity this actually represents a simplification of the given
generating system, a fact that is not obvious when only considering
graphs of groups, where the complexity is typically just the
complexity of the underlying graph.

\section{Graphs of graphs}
\label{sec::graphsofgraphs}

A \emph{square complex} is a two-dimensional cube complex. A square
complex is a \emph{$\mathcal{VH}$--complex}
(\cite[Definition~4.1]{wise}) if the edges may be identified as either
\emph{horizontal} or \emph{vertical}, such that the edges of each
square alternate between horizontal and vertical. If $X$ is a
$\mathcal{VH}$--complex we call a hyperplane \emph{vertical} if each
intersection with a square is parallel to vertical edges. A
$\mathcal{VH}$--complex is a \emph{graph of graphs} if each vertical
hyperplane has a product collar neighborhood. A \emph{morphism of
  graphs of graphs} is a continuous map which preserves the
$\mathcal{VH}$ structure and sends vertices to vertices and interiors
of edges to interiors of edges.

The connected components of the union of vertical edges of a graph of
graphs are \emph{vertex spaces}, and each vertical hyperplane is an
\emph{edge space}.  A graph of graphs $X$ always has an
\emph{underlying graph}, $\Gamma_X$, which is obtained by collapsing
each vertex space to a point and projecting squares to their
horizontal edges. If $v$ is a vertex of $\Gamma_X$ we call the
associated vertex space $X_v$, and if $e$ is an edge of $\Gamma_X$ we
call the associated edge space $X_e$.

Since $X_e$ has a product collar neighborhood, there are natural
\emph{edge maps} $\tau\colon X_e\to X_{\tau(e)}$ and $\iota\colon
X_e\to X_{\iota(e)}$ induced by the projections to vertical edges; $X$
is completely determined by the collection $(\Gamma_X,\{\tau\colon
X_e\to X_{\tau(e)}\})$, where $e$ varies over all oriented edges in
$\Gamma_X$. Note that $\tau\colon X_e\to X_{\tau(e)}$ is the same map
as $\iota\colon X_{\overline e}\to X_{\iota(\overline e)}$. A graph of
graphs in this sense is a functor from a graph in the second sense to
the category whose objects are graphs and whose morphisms are
morphisms of graphs which map interiors of edges homeomorphically to
interiors of edges. We require that $\overline{\phantom{X}}\colon
X_e\to X_e$ be the identity map.

A morphism of graphs of graphs $\varphi\colon X\to X'$ induces a
morphism of underlying graphs $\Gamma_X\to\Gamma_{X'}$. Denote the
restriction $\varphi\vert_{X_v}\colon X_v\to X'_{\varphi(v)}$ by
$\varphi_v$, likewise for edge spaces: $\varphi\vert_{X_e}$ is denoted
$\varphi_e$.

A graph of graphs \emph{over $Y$} is simply a graph of graphs with a
map $X\to Y$. Suppose $(X,\varphi)$ and $(X',\varphi')$ are graphs of
graphs over $Y$. A morphism $\sigma\colon X\to X'$ a \emph{morphism of
  graphs of graphs over $Y$} if $\varphi'\circ\sigma=\varphi$. Rather
than work with arbitrary generating systems for $\pi_1(Y)$ we consider
the category whose objects are graphs of graphs over $Y$ and whose
morphisms are morphisms of graphs of graphs over $Y$.

Let $X$ be a graph of graphs and let $(\Gamma_X,\{\tau\colon X_e\to
X_{\tau(e)}\})$ be the data determining $X$. There is an associated
graph of groups $\Delta_X$ determined by $(\Gamma_X,\{\tau_*\colon
\pi_1(X_e)\to \pi_1(X_{\tau(e)})\})$. If $\sigma\colon X\to Y$ is a
morphism of graphs of graphs then there is an obvious associated
morphism $\Delta_X\to\Delta_Y$.

\section{Surfaces are graphs of graphs}
\label{sec::surfaces}

To show that closed surfaces of nonpositive Euler characteristic are
graphs of graphs it suffices to find, for each $n\leq 0$, a graph
$V_n$ with $\chi(V_n)=n$, and a pair of immersions $\iota,\tau\colon
S^1\to V_n$ of the same length, such that every edge of $V_n$ is the
image of exactly two edges under $\iota$ and $\tau$.

\begin{lemma}
  Let $\Sigma$ be a closed surface with $\chi(\Sigma)\leq 0.$ Then
  $\Sigma$ may be given the structure of a graph of graphs with
  exactly one vertex space. In particular each vertex space has at
  least two incident edge spaces.
\end{lemma}

\begin{proof}
  See Figures~\ref{surfaceisgraphofgraphs} and~\ref{fig::glue}. Let $S_0$
  be either $A$ (even $\chi(\Sigma)$) or $M$ (odd $\chi(\Sigma)$), and
  let $R_0$ be the rectangle labeled $R$, and set $S_i=(S_{i-1}\sqcup
  T)/(R=R')$, identifying squares as indicated in the figure. Finally,
  set $R_i=N$. There is a retraction of $S_i$ to the core graph $V$
  (indicated by the thin black lines). The solid and dashed curves,
  $b_1$ and $b_2$, become immersions with the same length under the
  retraction, and each edge of the core graph is traversed twice by
  $b_1$ and $b_2$. Glue the boundary of a subdivided annulus $A$ to
  $V$ via $b_1$ and $b_2$.
\end{proof}

For the remainder of the paper, ``surface'' will mean topological
surface with a fixed graph of graphs structure such that each edge
space is nonseparating. 

\begin{figure}[ht]
\labellist
\pinlabel {$A$} [br] at 32 163
\pinlabel {$M$} [br] at 317 155
\pinlabel {$T$} [bl] at 608 170
\pinlabel \rotatebox{-90}{$R$} at 448 104
\pinlabel \rotatebox{-90}{$R$} at 161 94
\pinlabel {$R'$} at 686 89
\pinlabel {$N$} at 503 115
\endlabellist
\centerline{%
  \includegraphics[width=.95\textwidth]{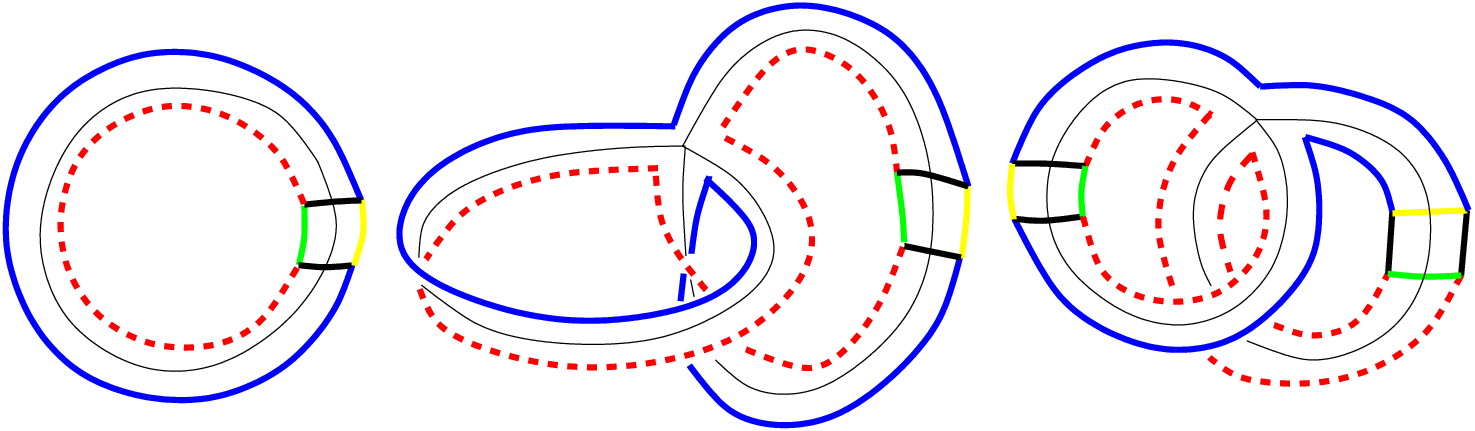}
}
\caption{The building blocks for surfaces.}
\label{surfaceisgraphofgraphs}
\end{figure}
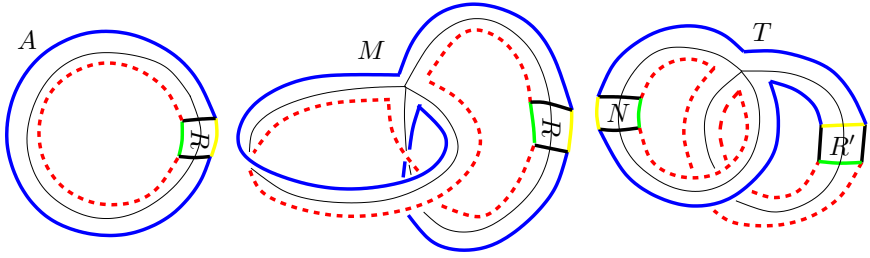

\begin{figure}[ht]
\labellist
\pinlabel {$S_{i-1}\sqcup T$} [b] at 163 155
\pinlabel {$S_i$} [bl] at 440 155
\pinlabel {$R_i$} at 334 95 
\pinlabel {$N$} at 20 95
\endlabellist
\centerline{%
  \includegraphics[width=.95\textwidth]{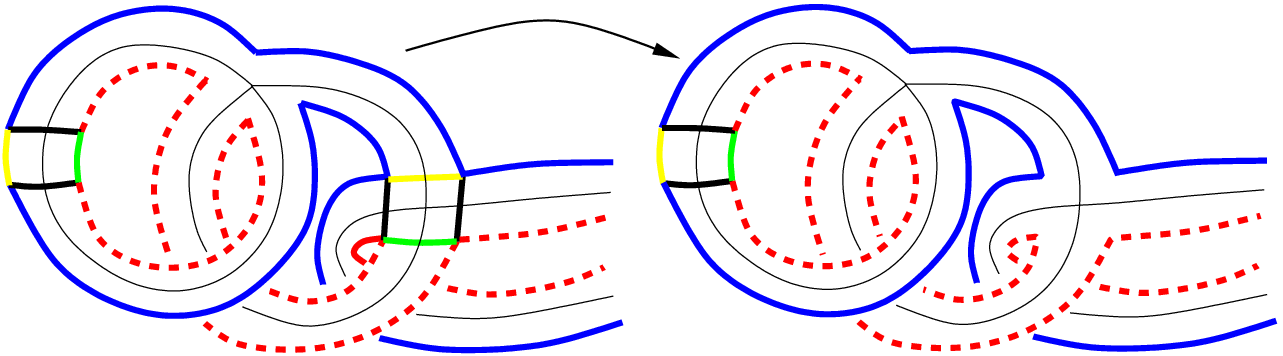}
}
\caption{Gluing the building blocks for surfaces together.}
\label{fig::glue}
\end{figure}

Let $\Sigma$ be a closed surface and let $\Sigma'$ be the graph of
graphs obtained by replacing an edge space $\Sigma_e$ by a vertex
$w\in\Sigma_e$. Then for some $p$, $\Sigma'\subset\Sigma$ is a strong
deformation retract of $\Sigma\setminus p$ and $(\Sigma',i)$ is a
minimal generating system.

\section{Coverlike graphs of graphs}

\begin{definition}
  Let $(X,\varphi)$ be a graph of graphs over a surface $\Sigma$. We
  say that $\varphi$ \emph{folds} the squares $s\neq s'$ if $s$ and
  $s'$ share an edge and have the same image in $Y$.  A graph of
  graphs $(X,\varphi)$ over $\Sigma$ is \emph{coverlike} if
  \begin{itemize}
  \item the one-skeleton of $X$ doesn't have any valence one vertices,
  \item $\varphi$ doesn't fold squares, and
  \item if $X_v\to \Sigma_{\varphi(v)}$ is not injective on
    fundamental group then $(X_v,\mathcal{E})$, where $\mathcal{E}$ is
    the collection of incident edge maps which are circular, is
    collapsible.
  \end{itemize}
\end{definition}

The second condition implies, in particular, that the restrictions
$\varphi_e\colon X_e\to\Sigma_{\varphi(e)}$ are immersions. Since edge
spaces in $\Sigma$ are circles this implies that edge spaces in $X$
are either circles or intervals.

\begin{definition}
  A map $\varphi\colon X\to\Sigma$ is \emph{pointwise injective} if,
  for all vertices $v$ of $\Gamma_X$, $\varphi_v$ is $\pi_1$
  injective. We say that $\varphi\colon X\to \Sigma$ is \emph{an
    immersion at $v$} if $\varphi_v\colon X_v\to \Sigma_{\varphi(v)}$
  is an immersion of graphs.  If $\varphi$ is coverlike and an
  immersion at all vertices then $\varphi$ is a \emph{pointwise
    immersion}.
\end{definition}

Let $\Sigma$ be a surface, and suppose $\varphi\colon X\to \Sigma$ is
coverlike and an immersion at $\tau(e)$. Let $\tau\colon X_e\to
X_{\tau(e)}$ be an edge map. There is
an edge space $\Sigma_{\varphi(e)}$ incident to
$\Sigma_{\varphi(\tau(e))}$ such that
\[\varphi_v\circ\tau_e=\tau_{\varphi(e)}\circ\varphi_e\]
The edge space $X_e$ is connected there is therefore a unique
connected component $\delta X_e$ of
$X_{\tau(e)}\times_{\Sigma_{\varphi(\tau(e))}}\Sigma_{\varphi(e)}$
containing the image of $X_e$.

Suppose $X_e$ is a point. Let $Y$ be the graph of graphs obtained by
replacing $X_e$ by $\delta X_e$ and $X_{\iota(e)}$ by
$X_{\iota(e)}\vee_{\iota(X_e)}\delta X_e $. There is a natural map
$\psi\colon Y\to\Sigma$, and $\varphi$ factors as the inclusion
$X\into Y$ followed by $\psi$. See Figure~\ref{fig::pulling}.

\begin{definition}
\label{def::pulling}
The space $Y$ is said to be obtained by \emph{pulling
  $\delta X_e$ across $e$}.
\end{definition}

The inclusion $X\into Y$ induces a morphism of graphs of groups
$\Delta_X\to\Delta_Y$ which is given by pulling a cyclic edge group
across a trivial edge. 

\begin{figure}[ht]
  \centerline{\includegraphics[width=.9\textwidth]{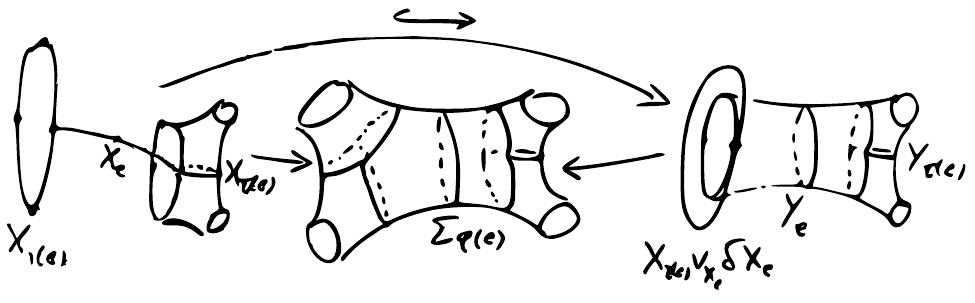}}
  \caption{Pulling $\delta X_e$ across $e$. In this case
    $\delta X_e$ is a circle.}
  \label{fig::pulling}
\end{figure}

The \emph{complexity} of a coverlike graph of graphs $(X,\varphi)$
over $\Sigma$ is the lexicographically ordered pair $\Comp(X)\define
(\rk(X),\vert X\vert)$, where $\vert X\vert$ is the number of edge
spaces in $X$. The complexity $\Comp(X)$ turns the set of coverlike graphs
of graphs into a poset. If $(X,\varphi)$ and $(X',\varphi')$ are
graphs of graphs over $Y$ and $\sigma\colon X\to X'$ is a reduction,
i.e., $\rk(X)>\rk(X')$, we write $X\succ X'$, and if $\sigma$ is an
equivalence, i.e., $\rk(X)=\rk(X')$, and the induced map of underlying
graphs $\sigma\colon\Gamma_X\to\Gamma_{X'}$ is an isomorphism of
graphs we write $X\simeq X'$. In this case $X$ and $X'$ are equivalent
and have the same number of edge spaces.

\section{Proof of Theorem~\ref{maintheorem}}

The following four subsections contain the argument used to prove
Theorem~\ref{maintheoremii}. In
\ref{subsec::decreaseedges}-\ref{subsec::findingobvious} we describe
three ``moves'' which take a generating system as input and give as
output equivalent, or simpler, generating systems which more closely
approximate standard ones. We summarize the main steps below.

\begin{enumerate}
\item[\ref{subsec::decreaseedges}] In this subsection we find
  conditions which guarantee that, given a coverlike generating
  system, a new equivalent or reduced coverlike generating system with
  fewer edge spaces can be found. The moves in this section are
  applied as aggressively as possible. If ever, in the course of
  transforming graphs of graphs, either of Lemma~\ref{foldablepair} or
  Lemma~\ref{lem::twosamedelta} is applicable, it is applied, and the
  entire folding process is restarted with their output as the new
  input. 

  {\bf We note that the \emph{only} time a coverlike generating system
    is allowed to have an edge space which is not a point or a circle
    is in Lemma~\ref{foldablepair}. The moves in this subsection are
    all folds on the level of graphs of groups.}

\item[\ref{subsec::increasingcircular}] In this subsection, since we
  assume it isn't possible to reduce the number of edge spaces, we try
  to increase the number of circular edge spaces so that we eventually
  find a non-pointwise injective coverlike generating system. In this
  subsection all generating systems \emph{except} for the output,
  which is fed into~\ref{subsec::findingobvious}, are pointwise
  injective. On the level of graphs of groups, unless a fold occurs,
  all moves are pulling cyclic subgroups across trivial edge groups.
\item[\ref{subsec::findingobvious}] In this subsection we take as
  input the output of~\ref{subsec::increasingcircular} and, assuming
  the number of edge spaces cannot be decreased, produce a generating
  system with an ``obvious relation'' and fewer circular edge
  spaces. On the level of graphs of groups all moves in this section
  are either vertex morphisms which yield free quotients (unlike those
  appearing in Example~\ref{ex::naive}) or are unpulling cyclic edge
  groups to reveal trivial edge groups.
\item[\ref{subsec::structureofminimal}] In this subsection we show
  that minimal complexity coverlike generating systems with obvious
  relations and as few circular edge spaces as possible are
  essentially finite sheeted covers. We then argue that they represent
  standard generating systems, otherwise we could have used
  \S\ref{subsec::findingobvious} or \S\ref{subsec::decreaseedges} to
  find either a generating system of strictly lower complexity or a
  generating system of equal complexity, but with an obvious relation
  and strictly fewer circular edge spaces.
\end{enumerate}

\subsection{Decreasing the number of edge spaces}
\label{subsec::decreaseedges}

We say that an edge (either horizontal or vertical) $h$ in $X$ is a
\emph{free face} if $h$ is contained in \emph{exactly} one square. A
vertical edge $h$ is a free face if and only if it is traversed
exactly once by exactly one incident edge map, and a horizontal edge
is a free face if and only if it corresponds to a valence one vertex
of an edge space. If $a$ is a vertex in an edge space denote the
associated horizontal edge in $X$ by $\left[a\right]$. Note that all
horizontal edges correspond to vertices in edge spaces, and all
squares correspond to edges in edge spaces.

\begin{lemma}[Shenitzer Theorem]
  \label{lem::isfreeface}
  If $(X,\varphi)$ is free, doesn't fold squares, and contains a
  square, then $X$ has a free face.
\end{lemma}

\begin{proof}
  If $X$ doesn't have a free face then there is a closed surface $S$
  and an immersion $S\to X$ which is one-to-one on squares. Since $X$
  is free, $S$ is a sphere. Every vertex in $S^1$ has even valence,
  and by Gauss-Bonnet there is a valence two vertex where two squares
  $s$ and $s'$ meet; $s$ and $s'$ share adjacent vertical and
  horizontal edges, and therefore have the same image in $\Sigma$, but
  this implies that $\varphi$ folds $s$ and $s'$.
\end{proof}

Let $\left[b\right]$ be a horizontal free face in $X$. Let $Y\simeq X$
be the graph of graphs obtained by first restricting $X_e$ to $a$ and
then removing valence one vertices from the one skeleton. The
inclusion map $Y\into X$ is a homotopy equivalence, and we say that
$Y$ is obtained from $X$ by \emph{collapsing the edge space onto $a$.}
The inclusion $Y\into X$ has no effect on the associated graph of
groups decompositions: $\Delta_Y\to \Delta_X$ is an isomorphism of
graphs of groups. 

Let $f$ be a vertical free face traversed by an edge $f'$ in
$\tau\colon X_e\to X_v$, and let $a\in X_e$ be a vertex. Let $Y\simeq
X$ be the graph of graphs obtained by removing the interiors of $f'$
and $f$ from $X_e$ and $X_v$, respectively. Then $Y_e$ is either a
single vertex or an interval. In the latter case, collapse $Y_e$ onto
$a$ and remove valence one vertices from the one-skeleton to obtain
$Z\simeq Y$. We say that $Z$ is obtained by \emph{collapsing the
  annulus from $f$ to $a$}. The natural map $Z\into X$ is a homotopy
equivalence. The map $\Delta_Z\to \Delta_X$ is clearly given by
pulling a cyclic group across a trivial edge group, and $\Delta_Z$ is
therefore obtained by \emph{unpulling} a cyclic edge group to reveal a
trivial edge group.


\begin{lemma}
  \label{lem::shenitzer}
  Let $(X,\varphi)$ be as in Lemma~\ref{lem::isfreeface}. For each
  edge space $X_e$ let $a_e$ be a distinguished vertex and let
  $A=\cup_e\left[a_e\right]\subset X^1$. There is a graph $W\subset
  X^1$ containing $A$ such that $\vert A\vert=\vert W\vert=\vert
  X\vert$ and the inclusion $W\into X$ is a homotopy equivalence.
\end{lemma}

\begin{proof}
  The proof is by induction on the number of squares in $X$. If $X$
  doesn't have any squares there is nothing to prove, so suppose $X$
  contains a square. Then by Lemma~\ref{lem::isfreeface} there is a
  free face $h$.  If $h$ is horizontal then it corresponds to a
  valence one vertex in $X_e$ for some $e$; $X_e$ is an interval, and
  we collapse $X_e$ onto $a_e$. Repeat until there are no horizontal
  free faces. If $h$ is vertical, then it's in the image $X_e\to
  X_{\tau(e)}$ for some $e$, with $X_e$ circular. Collapsing the
  annulus from $h$ to $a_e$ preserves the number of edge spaces and
  reduces the number of circular edge spaces. Repeat until there are
  no squares remaining. The resulting graph is the desired $W$.
\end{proof}

\begin{lemma}
  \label{foldablepair}
  Let $(X,\varphi)$ be coverlike and free. Suppose $\varphi$ folds a
  pair of horizontal edges $\left[a_e\right]$ and $\left[a_{f}\right]$
  corresponding to distinct edge spaces $X_e$ and $X_f$ of $X$. Then
  there is a graph $Z$ over $\Sigma$ with $X\succ Z$.
\end{lemma}

\begin{proof}
  Complete $\{a_e,a_f\}$ to a collection of horizontal edges, one from
  each edge space, and let $W\subset X^1$ be as in
  Lemma~\ref{lem::shenitzer}. The induced map $W\to \Sigma$ folds
  $\left[a_e\right]$ and $\left[a_{f}\right]$. Fold them together to
  obtain $Z$. Clearly $\vert Z\vert<\vert X\vert$.
\end{proof}

Note that the map $W\to Z$ is an equivalence if and only if
$\left[a_e\right]$ and $\left[a_f\right]$ don't share both
endpoints. The morphism of associated graphs of groups
$\Delta_W\to\Delta_Z$ is a fold.

\begin{lemma}
  \label{lem::twosamedelta}
  Let $(X,\varphi)$ be coverlike, free, and a pointwise immersion. If
  there are edges $e$ and $f$ such that $\delta X_e$ and $\delta X_f$
  agree then there is a graph $Z$ with $X\succ Z$.
\end{lemma}

\begin{proof}
  If $X_e$ is a circle, and $\delta X_e$ and $\delta X_f$ agree then
  $X_f$ is a point, otherwise $X\to\Sigma$ folds squares. Then
  $\varphi$ folds horizontal edges and, by Lemma~\ref{foldablepair},
  there is a graph $Z$ such that $X\succ Z$.

  We may therefore assume that both $X_e$ and $X_f$ are points, and
  that if $\delta X_g=\delta X_f$ then $X_g$ is a point. Pull $\delta
  X_e$ across $e$ to obtain a new coverlike $(Y,\psi)$ equivalent to
  $X$. Now $\psi$ folds horizontal edges in distinct edge
  spaces. Apply Lemma~\ref{foldablepair}.
\end{proof}

\subsection{Increasing the number of circular edge spaces}
\label{subsec::increasingcircular}

Given a pointwise immersion $X\to\Sigma$ we can either find a
reduction of $X$ or make $X$ look more like a surface by increasing
the number of circular edge spaces.

\begin{lemma}
  \label{lem::makenotinjective}
  Let $(X,\varphi)$ be coverlike, free, and a pointwise
  immersion. Then either
  \begin{itemize}
  \item there is a graph $Z$ with $X\succ Z$ or
  \item $X\simeq X'$ with a collapsible vertex space $X'_v$ such that
    the map $\varphi_v\colon X'_v\to\Sigma_{\varphi'(v)}$ is not
    $\pi_1$--injective.
  \end{itemize}
\end{lemma}

First we prove some auxiliary lemmas.

\begin{lemma}[\emph{cf.} {\cite[Theorem~2.1]{dunwoody}, \cite[Lemma~3.1]{stallings::binding}}]
  \label{lem::stallingsdunwoody}
  Let $\varphi\colon X\to\Sigma$ be a pointwise immersion. If
  $\varphi$ isn't $\pi_1$ injective then either
  \begin{itemize}
  \item There is a vertex $v$ and distinct edge spaces $X_e$ and $X_f$
    incident to $X_v$ such that $\delta X_e=\delta X_f$, or
  \item There is an edge space $X_e$ which is a point, such that
    $\delta X_e$ is a circle.
  \end{itemize}
\end{lemma}

The proof is an adaptation of Stallings' method of binding ties to
morphisms of coverlike graphs of graphs, and is the analogue for
pointwise immersions of Stallings' observation that immersions of
graphs are $\pi_1$--injective. Roughly, the lemma says that if no
element of the fundamental group of a vertex space dies (guaranteed by
$\varphi$ being a pointwise immersion), and if no two distinct edge
spaces can be folded together (bullet one), then it's possible to
enlarge an edge space to a circle (bullet two). Roughly, in the
language of graphs of groups, if the map $\Delta_X\to\Delta_{\Sigma}$
doesn't fold distinct edge groups, and if there are no possible vertex
morphisms, then a cyclic group can be pulled across a trivial edge
group.


\begin{proof}[Proof of Lemma~\ref{lem::stallingsdunwoody}.]
  Assume the first case doesn't hold. Let $c\colon \mathbb{S}^1\to X$
  be a shortest edge path in the one-skeleton of $X$ such that
  $\varphi\circ c$ is nullhomotopic in $\Sigma$, and let $h\colon
  D\to\Sigma$ be a map of a least area square singular planar square
  complex homotopy equivalent to a disk such that $h\vert_{\partial
    D}$ lifts to $c$. Note that if there is a curve $c$ bounding a
  disk $D$ with no area then $c$ is an edge path in $X_v$ for some $v$
  and the map $X_v\to\Sigma_{\varphi(v)}$ isn't injective, contrary to
  hypothesis. We may also assume that if two squares in $D$ share an
  edge then they are not folded together by $h$, otherwise we may find
  a disk of lower area (perhaps bounding a shorter curve with
  nullhomotopic image). Let $\Lambda$ be the preimage of the edge
  spaces of $\Sigma$ under $h$. Since $D$ has least area, the
  connected components of $\Lambda$ are arcs connecting $\partial D$
  to itself. Furthermore, $h\vert_{\Lambda}\colon\Lambda\to
  \sqcup_e\Sigma_e$ is an immersion. Let $\lambda$ be an outermost arc
  in $\Lambda$. See Figure~\ref{fig::outermostarc}. Let $a$ and $b$ be
  the vertices of $X_e$ and $X_f$ corresponding to the endpoints of
  $\lambda$. Since $\lambda$ is connected, and by definition of the
  pullback, $\delta X_e=\delta X_{f}$. Since we aren't in the first
  case of the lemma $a=b$, $e=f$, the endpoints of $\lambda$ coincide
  and $\delta X_e$ is a circle. If $X_e$ is a circle, let $D'$ be the
  disk obtained by removing squares meeting $\lambda$. Then, since
  $a=b$, the boundary of $D'$ lifts to a shorter curve with
  nullhomotopic image, contradicting our choice of $c$.
\end{proof}

\begin{figure}
\centerline{\includegraphics[width=.5\textwidth]{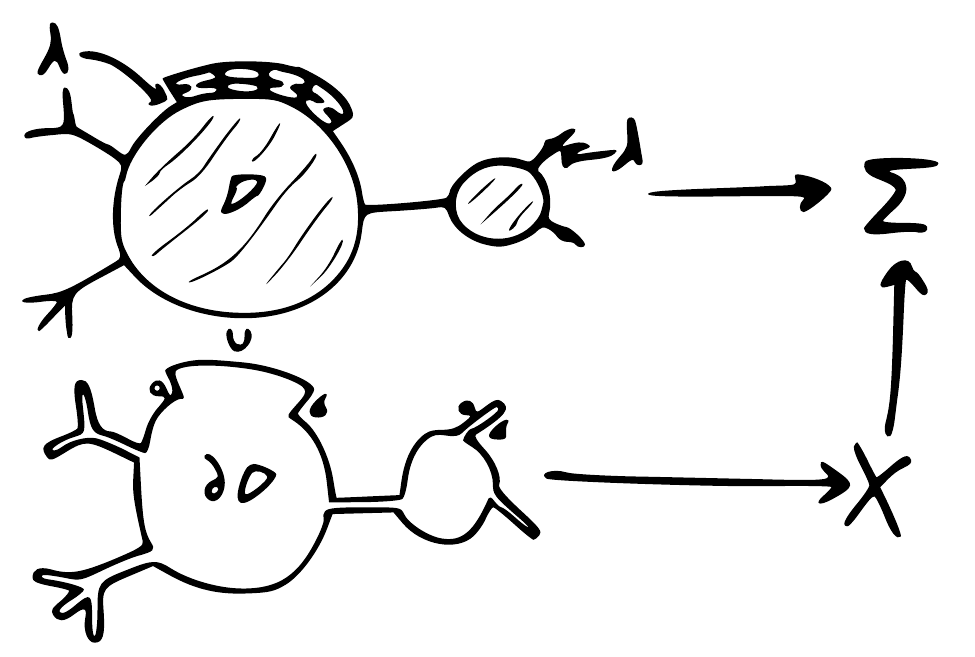}}
\caption{An outermost arc $\lambda$ in $D$.}
\label{fig::outermostarc}
\end{figure}

\begin{lemma}
  \label{lem::pullingfd}
  Let $(Y,\varphi)$ be a pointwise immersion, and suppose that
  Lemma~\ref{lem::twosamedelta} is not applicable, i.e., if $\delta
  Y_e$ and $\delta Y_f$ agree then $e=f$. Then for some $e$, $Y_e$ is
  a point, $\delta Y_e$ is a circle, and the graph of graphs $X$
  obtained by pulling $\delta Y_e$ across $e$ is
  coverlike. Furthermore, $\vert Y\vert_c<\vert X\vert_c$.
\end{lemma}

Note that we have swapped the roles of $Y$ and $X$ to keep the
notation consistent with the way Lemma~\ref{lem::pullingfd} is used in
the proof of Lemma~\ref{lem::makenotinjective}.

\begin{proof}
  By Lemma~\ref{lem::stallingsdunwoody} there is an edge space $Y_e$
  which is a point with $\delta Y_e$ a circle. Suppose
  $X_{\iota(e)}$ is not collapsible. Since
  $X_{\iota(e)}=Y_{\iota(e)}\vee_{\iota(Y_e)}\delta Y_e$, there is some
  subgraph $W$ of $Y_{\iota(e)}$ for which every edge is crossed by
  two circular incident edge maps. Since $Y\to\Sigma$ is a pointwise
  immersion this implies $Y_{\iota(e)}\to\Sigma$ is a finite sheeted
  cover and that $W=Y_{\iota(e)}$. Since $Y_e$ is a point there is an
  incident edge map $Y_f\to Y_{\iota(e)}$ such that $\delta
  Y_{\overline e}$ and $\delta Y_f$ agree, contradicting the
  assumption that Lemma~\ref{lem::twosamedelta} is not applicable. See
  Figure~\ref{fig::ifnotcollapsible}. Hence $X$ is coverlike. Clearly
  $\vert Y\vert_c<\vert X\vert_c$.
\end{proof}

\begin{figure}
  \centerline{\includegraphics[width=.8\textwidth]{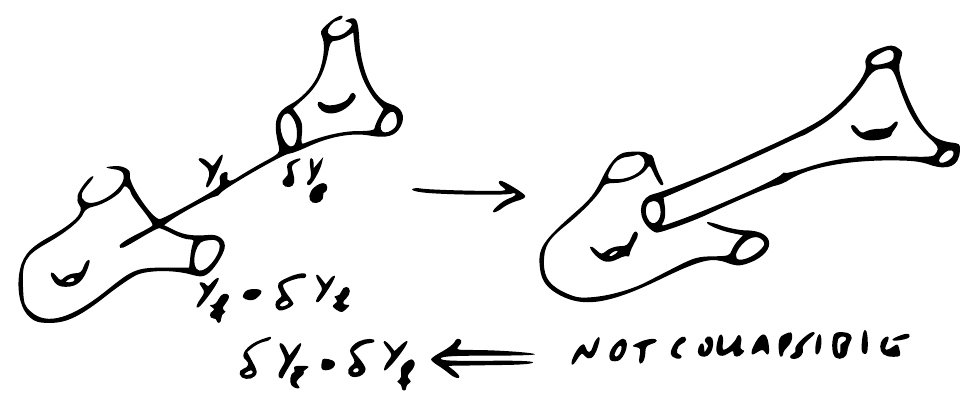}}
  \caption{If $X$ is not coverlike then distinct edge spaces can be
    folded together.}
  \label{fig::ifnotcollapsible}
\end{figure}

Suppose $X\to \Sigma$ is pointwise $\pi_1$ injective but not an
immersion at some vertex $v$. The vertex map $\varphi_v\colon X_v\to
\Sigma_{\varphi(v)},$ by Theorem~\ref{stallings}, factors through a
folding sequence
\[
X_v=W_0\fontop{\phi_0}W_1\fontop{\phi_1} W_2\to\dotsb\to W_k\immerses Y_{\varphi(v)}
\]
By replacing $X_v$ by $W_i$ and composing maps, we obtain a sequence
of graphs of graphs $(X^i,\varphi_i)$ over $\Sigma.$ At each stage,
there is a morphism $\varepsilon_i\colon X^i\to X^{i+1}$ induced by
$\phi_i.$ If $\varepsilon_i$ is a homotopy equivalence then $\phi_i$
is, and vice-versa.  The last map $\varphi_k\colon X^k\to Y$ is an
immersion at $v$. Folding in vertex spaces may introduce valence one
vertices to the one-skeleton. We remove valence one vertices and
incident edges without comment.

\begin{lemma}
  \label{lem::foldingkeepscoverlike}
  If $(X,\varphi)$ is coverlike and $\varphi_v$ is $\pi_1$ injective
  then folding $X_v$ to an immersion gives a coverlike generating
  system $(Y,\psi)$.
\end{lemma}

Recall that we are assuming in this section (and everywhere outside
Lemma~\ref{foldablepair}) that all edge spaces are either points or
circles.

\begin{proof}
  Let $\mathcal{E}$ be the collection of circular incident edge spaces
  at $v$.

  If $(X_v,\mathcal{E})$ is not reducible there is a connected
  subgraph $W\subset X_v$ which contains the image of each element of
  $\mathcal{E}$, such that $W\to\Sigma_{\varphi(v)}$ is a finite
  sheeted cover, $X_v$ is $W$ with finitely many finite trees
  attached, and $X_v$ folds down to $W$. If any squares fold then we
  would have had to fold a pair of edges in $W$, which we didn't.

  If $(X_v,\mathcal{E})$ is reducible and $Y$ is not coverlike then
  there are two circular edge spaces $X_e$ and $X_f$ incident to $X_v$
  which don't fold in $X$ (hence they represent distinct conjugacy
  classes of maximal cyclic subgroups in $X_v$) such that $Y_e$ and
  $Y_f$ fold in $Y$ (hence represent indistinct conjugacy classes of
  maximal cyclic subgroups in $Y_v$). Since
  $Y_v\to\Sigma_{\varphi(v)}$ is an immersion, $Y_e$ and $Y_f$ have
  the same image in $Y$, contradicting
  Lemma~\ref{lem::foldcollapsible} applied to $(X_v,\mathcal{E})$.
\end{proof}

\begin{proof}[Proof of Lemma~\ref{lem::makenotinjective}]

  Set $X^0=X$. Suppose $X^i$ has been defined. If $X^i\to\Sigma$ is
  not pointwise injective set $X'=X^i.$ Otherwise fold vertex spaces
  to immersions to obtain $Y$. Then by
  Lemma~\ref{lem::foldingkeepscoverlike} $Y\to\Sigma$ is
  coverlike. (This step is redundant for $i=0$, as $X\to\Sigma$ is
  already a pointwise immersion.) If $Y_e$ and $Y_f$ are incident to
  $Y_v$ and $\delta Y_e=\delta Y_f$ then Lemma~\ref{lem::twosamedelta}
  is applicable and there is a graph $Z\prec Y$. Otherwise, by
  Lemma~\ref{lem::pullingfd} there is an edge $e$ such that $Y_e$ is a
  point, $\delta Y_e$ is a circle and $Y_{\iota(e)}$ is
  collapsible. Let $X^{i+1}$ be the space obtained by pulling $\delta
  Y_e$ across $e$. Since $\vert X^i\vert_c<\vert
  X^{i+1}\vert_c\leq\vert X\vert$, for some $i$, $X^{i+1}\to\Sigma$ is
  not pointwise injective.
\end{proof}


\subsection{Finding obvious relations}
\label{subsec::findingobvious}

\begin{definition}
  Suppose $X\to\Sigma$ is coverlike, free, and a pointwise
  immersion. If $X_e$ is a point and $\delta X_e$ and $\delta
  X_{\overline e}$ are both circles then $X$ \emph{has an obvious
    relation}.
\end{definition}

Generating systems $X\to\Sigma$ which aren't pointwise injective are
particularly important. In this case we fold pairs of vertical edges
so long as the induced graphs of graphs stay in the category of
coverlike generating systems. At some point in the folding sequence
the induced maps of graphs of graphs either fold horizontal edges or
aren't homotopy equivalences. In the first case we can either reduce
the number of edge spaces (first bullet) or find obvious relations
(second bullet), and in the second case we argue that the generating
system was either reducible (first bullet), equivalent to one with an
obvious relation and fewer circular edge spaces (second bullet), or
equivalent to a standard one to begin with (third bullet).

\begin{lemma}
  \label{lem::npithenor}
  Let $(X,\varphi)$ be coverlike, and suppose that $\varphi$ is not
  pointwise injective at exactly one vertex space. Then either
  \begin{itemize}
  \item there is a graph $Z$ with $X\succ Z$ 
  \item $X\simeq X'$ with an obvious relation and $\vert X'\vert_c<
    \vert X\vert_c$ or
  \item $X\simeq Y$ with $Y$ standard.
  \end{itemize}
\end{lemma}

Hilfssatz~5 of~\cite{zieschang} is a (very) special case of this
lemma. We first prove some auxiliary lemmas.

\begin{lemma}
  \label{lem::ifbranchedcover}
  If $\varphi\colon X\to \Sigma$ is a branched cover then $\varphi$
  is either a cover or folds horizontal edges.
\end{lemma}

\begin{proof}
  Suppose $\varphi$ is not a cover, and let $p$ be a nontrivial branch
  point in $X$. A neighborhood of $p$ maps to a neighborhood of
  $\varphi(p)$ like the map $z\mapsto z^k$ for some $k>1$. If $e$ is a
  horizontal edge incident on $\varphi(p)$ then there are edges
  $e_1,\dotsc,e_k$ incident to $p$ which map to $e$, but this is
  precisely what it means for $\varphi$ to fold horizontal edges.
\end{proof}



Coverlike generating systems $X\to \Sigma$ which fold edges in
distinct horizontal edge spaces are either reducible or equivalent to
coverlike generating systems with fewer edge spaces. The situation is
somewhat different for generating systems which fold horizontal edges
from the same edge space. Suppose $X\to\Sigma$ is coverlike. We say
that $X_e$ \emph{self-folds} if there are vertices $a$ and $b$ in
$X_e$ such that $\varphi(a)=\varphi(b)$ and $\tau(a)=\tau(b)$. The
edges $\left[a\right]$ and $\left[b\right]$ share endpoints and have
the same image in $\Sigma$. 

\begin{lemma}
  \label{lem::selffoldnotinjective}
  Let $(X,\varphi)$ be coverlike, and suppose that $X_e\to X_v$
  self-folds. Then the map $X_v\to\Sigma_{\varphi(v)}$ is not
  injective on fundamental group.
\end{lemma}

\begin{proof}
  We keep the notation from above. Since $\varphi$ folds
  $\left[a\right]$ and $\left[b\right]$, we have
  $\varphi(a)=\varphi(b)$ and $\tau(a)=\tau(b)$. Hence both maps
  $X_e\to X_v$ and $X_e\to\Sigma_{\varphi(e)}$ factor through the
  quotient $R=X_e/\{a\sim b\}$. The map $R\to X_v$ has nonabelian
  image, but the maps $R\to X_v\to \Sigma_{\varphi(v)}$ and $R\to
  \Sigma_{\varphi(e)}\to \Sigma_{\varphi(v)}$ agree, but
  $\Sigma_{\varphi(e)}$ is circular, therefore the nonabelian subgroup
  of $X_v$ carried by $R$ has abelian image in $\Sigma_{\varphi(v)}$
\end{proof}

\begin{lemma}
  \label{lem::foldinsameedgespace}
  Suppose that $\varphi$ folds a pair of horizontal edges
  $\left[a\right]$ and $\left[b\right]$, with $a,b\in X_e$ and
  $\tau(a)=\tau(b)$, that $Y$ is obtained from $X$ by an annulus
  collapse from $v$, and that $Y\to\Sigma$ is pointwise
  injective. Then $Y$ folds down to $Z\to\Sigma$ with an obvious
  relation.
\end{lemma}

\begin{proof}
  Let $h$ be the edge participating in the annulus collapse from
  $v$. By Lemma~\ref{lem::selffoldnotinjective} the annulus collapse
  must correspond to $X_e$, otherwise the map $Y_v\to
  \Sigma_{\varphi(v)}$ isn't $\pi_1$ injective.

  Let $f$ be the single edge of $X_e$ that maps to $h$, and write
  $X_e\to X_v$ as a composition of reduced edge paths $S_1S_2$, where
  $S_1$ traverses $X_e$ from $a$ to $b$, missing $f$, and $S_2$ from
  $b$ back to $a$, traversing $f$. Since $Y\to\Sigma$ is pointwise
  injective and doesn't fold squares, it folds down to a pointwise
  immersion $Z\to\Sigma$. Then $\delta Z_e=S_1/\{a\sim b\}$ and
  $\delta Z_{\overline{e}}=\delta X_{\overline{e}}$. See
  Figure~\ref{fig::unpulltoor}
\end{proof}  

\begin{figure}[ht]
\centerline{
  \includegraphics[width=.8\textwidth]{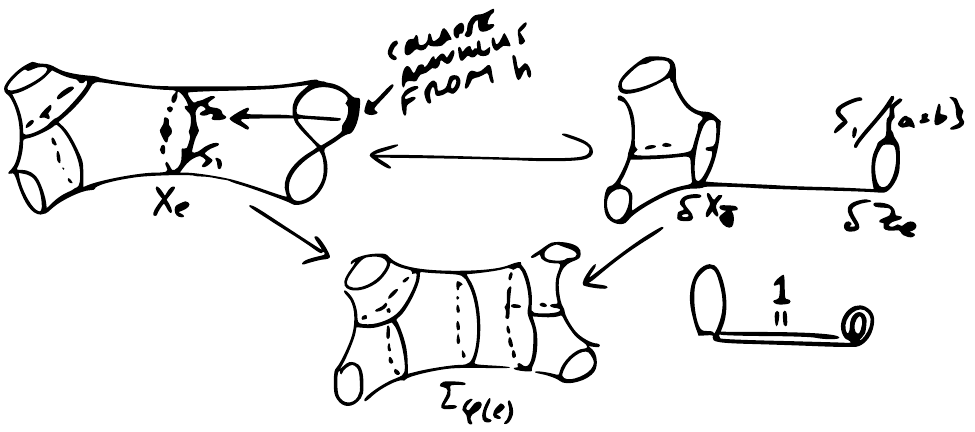}
}
\caption{Collapsing an annulus to reveal an obvious relation.}
\label{fig::unpulltoor}
\end{figure}


Let $\varphi\colon X\to\Sigma$ be injective at $v\in\Gamma_X$, and
let $f_1,\dotsc,f_n$ be the oriented edges of $\Gamma_{\Sigma}$
incident to $w=\varphi(v)$. Suppose that $X_v$ contains a subgraph $W$
such that the map $W\into X_v\to\Sigma_{\varphi(v)}$ is a finite
sheeted cover. Let $e^i_j$ be the oriented edges of $\Gamma_X$
incident to $v$ such that $\varphi(e^i_j)=f_i$ and $X_{e^i_j}$ has
image contained in in $W$. We say that $X_v$ is \emph{full} if, for
each $j$, the natural map
\[
\sqcup_j X_{e^i_j}\to X_{v}\times_{\Sigma_w}\Sigma_{f_i}
\]
is surjective on the set of connected components. Note that if
$X_v\to\Sigma_{\varphi(v)}$ is $\pi_1$--injective then there is no
choice for $W$.

\begin{lemma}
  \label{lem::isfull}
  Let $(X,\varphi)$ be coverlike, free, injective at $v$, and suppose
  $X_v$ is full. Let $W\subseteq X_v$ be as above. If $W\neq X_v$ then
  there is a graph $Z$ with $X\succ Z$.
\end{lemma}

\begin{proof}
  Since $W\to \Sigma_w$ is a finite sheeted cover and $X_v\to\Sigma$
  is injective on fundamental group, $X_v$ is obtained from $W$ by
  attaching a collection of finite trees. Since $X_v$ has valence one
  vertices and $X^1$ doesn't have valence one vertices, there is an
  edge $g$ of $\Gamma_U$ of $X$ such that $\tau\colon X_g\to X_v$ has
  image in $X_v\setminus W$. There is a folding sequence
  \[
  X_v\to\dotsb\to W\immerses\Sigma_w
  \]
  Let $\psi\colon Y\to\Sigma$ be the generating system obtained by
  quotienting $X$ by the composition $X_v\to W$. Clearly $Y_v=W$. Then
  there is some oriented edge $f_i$ incident to $w$ such that
  $\psi(g)=f_i$. Since
  \[
  \sqcup_j Y_{e^i_j}\to Y_{v}\times_{\Sigma_w}\Sigma_{f_i}
  \]
  is surjective the induced map
  \[
  Y_g\to Y_v\times_{\Sigma_w}\Sigma_{f_i}
  \]
  lands in a connected component which contains the image of
  $Y_{e^i_j}$ for some $i$. Since $g\neq e^i_j$, $Y_{e^i_j}$ and $Y_g$
  fold. Apply Lemma~\ref{lem::twosamedelta}.
\end{proof}

\begin{proof}[Proof of Lemma~\ref{lem::npithenor}]
  Since $\varphi$ is not pointwise injective, there is a collapsible
  vertex space $X_v$ such that $\varphi_v\colon
  X_v\to\Sigma_{\varphi(v)}$ is not $\pi_1$--injective.  Let
  \[
  X_v=V_0\to V_1\to\dotsb\to V_l
  \] 
  be a folding sequence for $\varphi_v\colon
  X_v\to\Sigma_{\varphi(v)}$, and let $X^i$ be the space obtained by
  replacing $X_v$ by $V_i.$ Let $k$ be the first index such that
  either $X^k\to\Sigma$ folds a pair of horizontal edges or $V_k\to
  V_{k+1}$ is not a homotopy equivalence. Note that $X^k\to\Sigma$ is
  coverlike: if not then it folds a pair of squares, but then
  $X^{k-1}\to\Sigma$ folds a pair of horizontal edges, contrary to our
  choice of $k$.

  If $X^k\to\Sigma$ folds horizontal edges in distinct edge spaces, by
  Lemma~\ref{foldablepair}, there is a graph $Z\prec X$. Suppose $X^k$
  has a self-folding edge space $X^k_e\to X^k_v$. Since $X^k_v$ is
  collapsible there is an edge space $X^k_f$ and an edge $h\subset
  X^k_v$ such that $X^k_f$ is the only incident edge space which
  traverses $h$. Let $Y$ be the coverlike graph of graphs obtained by
  collapsing the annulus from $h$. If $e=f$ and $Y$ is pointwise
  injective then by Lemma~\ref{lem::foldinsameedgespace} $Y$ folds
  down to a pointwise immersion with an obvious relation. If $Y$ is
  not pointwise injective then the map $Y_v\to\Sigma_{\varphi(v)}$ is
  not injective on fundamental group and we repeat the process
  starting with $Y$. Likewise, if $e\neq f$ then
  $Y_v\to\Sigma_{\varphi(v)}$ is not injective on fundamental group
  and we start over using $Y$ as the initial data. Since $\vert
  Y\vert_c<\vert X\vert_c$ this can happen only finitely many times.

  Thus we may assume that that $X^k\to\Sigma$ doesn't fold
  horizontal edges. Since $X^k_v\to X^{k+1}_v$ isn't a homotopy
  equivalence there is a pair of edges $g,h$ in $X^k_v$ which have the
  same endpoints and are identified by $\varphi$.

  Suppose that $X^{k+1}\to\Sigma$ folds squares. Then $X^{k+1}$ folds
  horizontal edges, but since $X^{k+1}$ is obtained by identifying
  edges in $X^k_v$, $X^{k+1}$ and $X^k$ have the same horizontal
  one-skeleton, hence $X^k\to\Sigma$ folds horizontal edges, contrary
  to the previous paragraph.


  If $\pi_1(X^{k+1})$ is free then by Lemma~\ref{lem::shenitzer}
  $X\simeq X^k\succ X^{k+1}\succeq Z$, where $Z$ is a graph, hence the
  first bullet holds.

  Hence we assume $\pi_1(X^{k+1})$ is not free. Then there is a closed
  surface $S$ of non-positive Euler characteristic and an immersion
  $S\to X^{k+1}$ which is one-to-one on the set of squares. Since
  $X^{k+1}\to \Sigma$ doesn't fold squares, the map $S\to \Sigma$
  doesn't fold squares, and is therefore a branched cover by
  Lemma~\ref{lem::ifbranchedcover}. Since $X^{k+1}\to \Sigma$ doesn't
  fold horizontal edges, the map $S\to \Sigma$ is a cover and
  we conclude that the map $S\to X^{k+1}$ is an embedding, otherwise
  there is a pair of edges in $S$ which have the same image in
  $\Sigma$ and share a vertex in $X^{k+1}$.

  Collapse $X^k$ from $h$ to obtain $Y$. If $Y\to \Sigma$ isn't
  pointwise injective then start over using $Y$ as the initial
  data. If it is injective, since $S$ is a finite sheeted cover and
  $g$ and $h$ share endpoints and have the same image in $\Sigma$,
  $Y_v$ has a connected subgraph $W\subset Y_v=X^k_v\setminus
  h^{\circ}$ which is a finite sheeted cover of
  $\Sigma_{\varphi(v)}$. Then $Y_v$ is full, and if $W\subsetneq Y_v$
  then by Lemma~\ref{lem::isfull} there is a graph $Z\prec Y\simeq X$.
  The only remaining possibility is that $W=Y_v$, but then the map
  $S\into X^{k+1}$ is a homeomorphism and $Y\to\Sigma$ is a standard
  generating system.
\end{proof}




\subsection{Minimal complexity, obvious relation $\Rightarrow$ standard}
\label{subsec::structureofminimal}

Let $f\colon\free_n\to\pi_1(\Sigma)$ be a marking. Represent $f$ as a
map of graphs of graphs as follows. Let $Z$ be a rose with fundamental
group $\free_n,$ and identify the petals of $Z$ with the generators of
$\free_n.$ There is a map $\varphi\colon Z\to\Sigma^{1}$ inducing
$f$. We may assume, by subdivision and homotopy, that $\varphi$ is a
morphism of graphs. We regard $Z$ as a graph of graphs over $\Sigma$
by declaring the connected components of preimages of vertex spaces of
$\Sigma$ to be vertex spaces. The remaining edges are horizontal and
map to horizontal edges of $\Sigma$. Midpoints of horizontal edges are
edge spaces.

Lemmas~\ref{lem::makenotinjective} and~\ref{lem::npithenor} imply that the
collection of $X\preceq Z$ with obvious relations is not empty. Choose
$X\preceq Z$ with an obvious relation such that if $X'\preceq X$ then
$X'\simeq X$ and
\[
\vert X\vert_c=\min\{\vert X'\vert_c\mid X\simeq X'\mbox{ and }X'\mbox{ has an obvious relation.}\}
\]

\begin{lemma}
  \label{lem::lookslikecover}
  Let $X$ be as above, with $X_e$ a point and $\delta X_e$ and $\delta
  X_{\overline e}$ circular. Then each edge in a vertex space is
  covered exactly twice by the collection of incident edge maps or
  $\delta X_e$ and $\delta X_{\overline e}$.
\end{lemma}

\begin{proof}
  First, any free face is traversed either by $\delta X_e$ or $\delta
  X_{\overline e},$ otherwise an annulus collapse gives a generating
  system with an obvious relation and fewer circular edge spaces.

  Hence all vertex spaces other than $X_{\iota(e)}$ and $X_{\tau(e)}$
  are finite sheeted covers of their associated vertex spaces of
  $\Sigma$, and all their incident edge spaces are circles. Suppose
  that $\delta X_e$ traverses an edge $h$ of $X_{\tau(e)}$ exactly
  once, that $\delta X_{\overline e}$ doesn't traverse $h$, and that
  no incident edge space traverses $h.$ Form $X'$ by pulling $\delta
  X_e$ across $e$. Then $h$ is traversed once by $X'_e$, not traversed
  by any other incident edge spaces, and is thus a free face in $X'$.
  Annulus collapse from $h$ to obtain $Y$. Since both $\delta X_e$ and
  $\delta X_{\overline e}$ were circles, $Y_{\iota(e)}$ is
  collapsible and $Y_{\iota(e)}\to\Sigma_{\varphi(\iota(e))}$ is not
  $\pi_1$--injective. Clearly $\vert X\vert_c=\vert Y\vert_c$, but
  then Lemma~\ref{lem::npithenor} implies that $X\simeq X''$ with $\vert
  X''\vert_c<\vert X\vert_c$, contradicting minimality.

  Thus every edge traversed by $\delta X_e$ is traversed again by
  $\delta X_e$, $\delta X_{\overline e}$, or another incident edge
  space, hence $X_{\tau(e)}\to\Sigma_{\varphi(\tau(e))}$ and
  $X_{\iota(e)}\to\Sigma_{\varphi(\iota(e))}$ are finite sheeted
  covers.
\end{proof}

Schematically, $\varphi\colon X\to\Sigma$ has the form depicted in
Figure~\ref{fig::finitesheeted}.

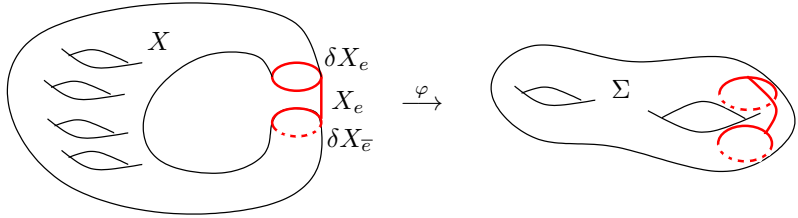
\begin{figure}[ht]
\labellist
\pinlabel $X$ at 58 66
\pinlabel $\Sigma$ at 232 48
\pinlabel $X_e$ [l] at 120 43
\pinlabel $\delta X_e$ [bl] at 118 53
\pinlabel $\delta X_{\overline e}$ [tl] at 119 35
\pinlabel $\fontop{\varphi}$ at 157 45
\endlabellist
\centerline{\includegraphics{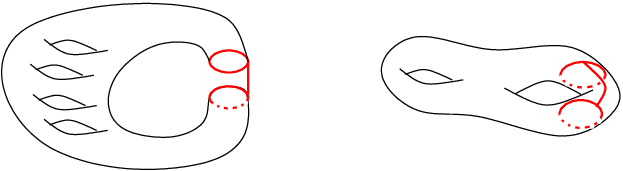}}
\caption{A candidate minimal complexity element with an obvious
  relation. $\varphi$ looks like a finite sheeted cover away from
  $X_e$.}
\label{fig::finitesheeted}
\end{figure}

\begin{lemma}
  \label{lem::noamalgam}
  $\delta X_{\overline e}$ traverses some edge of $X_{\iota(e)}$ that is
  traversed either by some other circular incident edge space or, if
  $\tau(e)=\iota(e)$, by some other circular incident edge space or
  $\delta X_e$.
\end{lemma}

\begin{proof}
  Suppose that every edge traversed by $\delta X_{\overline e}$ is
  traversed by $\delta X_{\overline e}$ twice. Then
  $\Sigma_{\varphi(\iota(e))}$ has only one incident edge map,
  contradicting our choice of graph of graphs structure on $\Sigma$.
\end{proof}



There are now two cases to consider, depending on the degrees of
$\delta X_e\to \Sigma_{\varphi(e)}$ and $\delta X_{\overline e}
\to\Sigma_{\varphi(\overline e)}=\Sigma_{\varphi(e)}.$ Suppose they
have different degrees. One of them, say the former, is smaller.  Let
$h$ be an edge of $X_{\iota(e)}$ traversed by $\delta X_{\overline e}$
and either by some other incident edge space or $\delta X_e.$ Pull
$\delta X_e$ across $e$ to form $X'$. In $X',$ $h$ is a free face and
is traversed by a circular incident edge space $X'_g$. Collapse the
annulus from $h$ in $X'$ to form $Z.$ By construction $Z$ is
coverlike, $\Comp(Z)=\Comp(X),$ $\vert X\vert_c=\vert Z\vert_c$, and
$Z_{\iota(e)}$ is collapsible. Since the degree of $\delta
X_e\to\Sigma_{\varphi(e)}$ is strictly less than the degree of $\delta
X_{\overline e}\to\Sigma_{\varphi(e)}$, Lemma~\ref{lem::notinjective}
implies the vertex map $Z_{\iota(e)}\to\Sigma_{\varphi(\iota(e))}$ is
not $\pi_1$--injective.

Since $\varphi_{\iota(e)}$ is not $\pi_1$--injective and $Z_g$ is not
circular, Lemma~\ref{lem::npithenor} gives a generating system with
fewer circular edge spaces and an obvious relation, contrary to
hypothesis.  Thus the two maps have the same degrees. Let $X'$ be the
space obtained by replacing $X_e$ by $\delta X_e$ or, equivalently,
since the degrees are the same, $\delta X_{\overline e}.$ In $X',$
every vertex map is finite index, all incident edge spaces are
circular, no horizontal edges fold, and therefore $X'\to\Sigma$ is a
finite sheeted cover. Since $\varphi_*(\pi_1(X))$ generates
$\pi_1(\Sigma),$ $X'\to\Sigma$ must be an isomorphism of square
complexes. Since $X$ is obtained by replacing $X'_e$ by a vertex, $X$
represents a minimal generating system. This completes the proof of
Theorem~\ref{maintheoremii}, hence Theorem~\ref{maintheorem} follows
as well.\hfill\qedsymbol

\bibliographystyle{amsalpha} \bibliography{nielsen}

\begin{flushleft}
\emph{email:} \texttt{l.louder@ucl.ac.uk, lars@d503.net}
\end{flushleft}

\end{document}